\documentclass[12pt]{amsart}
\usepackage{amssymb}
\usepackage{amsfonts}
\usepackage{amssymb,latexsym}
\usepackage{enumerate}
\usepackage{mathrsfs}
\makeatletter
\@namedef{subjclassname@2010}{%
  \textup{2010} Mathematics Subject Classification}
\makeatother

  [2002/01/19 v2.2g %
    AMS font definitions%
  ]
\DeclareFontFamily{U}{euf}{}
\DeclareFontShape{U}{euf}{m}{n}{%
  <5><6><7><8><9>gen*eufm%
  <10><10.95><12><14.4><17.28><20.74><24.88>eufm10%
  }{}
\DeclareFontShape{U}{euf}{b}{n}{%
  <5><6><7><8><9>gen*eufb%
  <10><10.95><12><14.4><17.28><20.74><24.88>eufb10%
  }{}

  [2002/01/19 v2.2g %
    AMS font definitions%
  ]
\DeclareFontFamily{U}{msb}{}
\DeclareFontShape{U}{msb}{m}{n}{%
  <5><6><7><8><9>gen*msbm%
  <10><10.95><12><14.4><17.28><20.74><24.88>msbm10%
  }{}

  [2002/01/19 v2.2g %
    AMS font definitions%
  ]
\DeclareFontFamily{U}{msa}{}
\DeclareFontShape{U}{msa}{m}{n}{%
  <5><6><7><8><9>gen*msam%
  <10><10.95><12><14.4><17.28><20.74><24.88>msam10%
  }{}

\newtheorem{theorem}{Theorem}[section]
\newtheorem{lemma}[theorem]{Lemma}
\newtheorem{proposition}[theorem]{Proposition}

\theoremstyle{definition}
\newtheorem{remark}[theorem]{Remark}

\newtheorem{example}[theorem]{Example}

\numberwithin{equation}{section} 

\def\C{\mathbb C_p}
\def\BZ{\mathbb Z}
\def\Z{\mathbb Z_p}
\def\Q{\mathbb Q_p}
\def\C{\mathbb C_p}
\def\BZ{\mathbb Z}
\def\Z{\mathbb Z_p}
\def\Q{\mathbb Q_p}

\begin{document}

\title[]
{On reciprocity formula of Apostol-Dedekind sum with quasi-periodic Euler functions}

\author{Su Hu}
\address{Department of Mathematics, South China University of Technology, Guangzhou, Guangdong 510640, China}
\email{hus04@mails.tsinghua.edu.cn}

\author{Daeyeoul Kim}
\address{National Institute for Mathematical Sciences \\ Yuseong-daero 1689-gil \\ Yuseong-gu \\
Daejeon 305-811 \\ South Korea} \email{daeyeoul@nims.re.kr}

\author{Min-Soo Kim}
\address{Center for General Education, Kyungnam University,
7(Woryeong-dong) kyungnamdaehak-ro, Masanhappo-gu, Changwon-si, Gyeongsangnam-do 631-701, Republic of Korea
}
\email{mskim@kyungnam.ac.kr}


\begin{abstract}
The Apostol-Dedekind sum
with quasi-periodic Euler functions is an analogue of Apostol's definition of the generalized Dedekind sum with periodic Bernoulli functions. In this paper, using  the Boole summation formula, we shall obtain the reciprocity formula for this sum.
\end{abstract}

\subjclass[2000]{11F20, 11B68, 65B15}
\keywords{Dedekind sum, Euler polynomial, Boole summation formula, Integral}


\maketitle

\def\B{\overline B}
\def\E{\overline E}
\def\Z{\mathbb Z_p}



\def\ord{\text{ord}_p}
\def\C{\mathbb C_p}
\def\BZ{\mathbb Z}
\def\Z{\mathbb Z_p}
\def\Q{\mathbb Q_p}
\def\wh{\widehat}
\def\ov{\overline}


\section{Introduction}
\label{Intro}

The Euler  polynomials $E_k(x),k\in\mathbb N_0=\mathbb N\cup\{0\},$ are defined by the generating function
\begin{equation}\label{Eu-pol}
\frac{2e^{xt}}{e^t+1}=\sum_{k=0}^\infty E_k(x)\frac{t^k}{k!}.
\end{equation}

The integers $E_{k}=2^{k}E_{k}\left({1}/{2}\right),k\in\mathbb N_0,$ are called Euler numbers.
For example, $E_0=1,E_2=-1,E_4=5,$ and $E_6=-61.$
The Euler numbers and polynomials (so called by Scherk in 1825) appear in Euler's famous book,
Insitutiones Calculi Differentials (1755, pp.487--491 and p.522).
Notice that the Euler numbers with odd subscripts vanish, that is, $E_{2m+1}=0$ for all $m\in\mathbb N_0.$
The Euler polynomials can also be expressed in terms of the Euler numbers in the following way (see \cite[p.~25]{No}):
\begin{equation}\label{E-nu-pol-re}
E_k(x)=\sum_{i=0}^k\binom ki \frac{E_i}{2^i}\left(x-\frac12\right)^{k-i}
\end{equation}
which holds for all nonnegative integers $m$ and all real $x,$ and
which was obtained by Raabe \cite{Ra} in 1851.

Some properties of Euler polynomials can be easily derived from their generating functions,
for example, from (\ref{Eu-pol}), we have
\begin{equation}\label{id-1}
x^k=\frac12(E_k(x+1)+E_k(x))\quad\text{and}\quad E_k(1-x)=(-1)^kE_{k}(x)
\end{equation}
(also see \cite[p.~530, (23) and (24)]{SC}).

For further   properties of the Euler polynomials and numbers including their applications, we refer to
\cite{AS,Ba,GR,No,Ra,SC}.
It may be interesting to point out that there is also a connection between the generalized Euler numbers and the ideal class group of the $p^{n+1}$th cyclotomic field when $p$ is a prime number.
For details, we refer to a recent paper~\cite{HK-I}, especially~\cite[Proposition 3.4]{HK-I}.

The $k$-th quasi-periodic Euler function $\overline{E}_k(x)$ is defined by (see \cite[p.~661]{Ca3})
\begin{equation}\label{ae-ft}
\E_k(x+1)=-\E_k(x)
\end{equation}
for all $x,$ and
\begin{equation}\label{ae-ft-1}
\E_k(x)=E_k(x) \quad\text{for} \quad 0\leq x<1,
\end{equation}
where $E_k(x),k\in\mathbb N_0,$ denotes the $k$th Euler polynomials.
It can be shown that $\E_k(x)$ has continuous derivatives up to the $(k-1)$st order.
For $x\in\mathbb R$,  $[x]$ denotes the greatest integer not exceeding $x$
and $\{x\}$ denotes the fractional part of real number $x,$ thus \begin{equation}\label{fr-in}\{x\}=x-[x].\end{equation}
Then, for $r\in\mathbb Z$ and $k\in\mathbb N_0,$ we have (see \cite[(1.2.9)]{Ba} and \cite[(3.3)]{Ca3})
\begin{equation}\label{ae-ft-p}
\E_k(x)=(-1)^{[x]}E_k(\{x\}),\quad \E_k(x+r)=(-1)^r\E_k(x).
\end{equation}

For further properties of the quasi-periodic Euler functions, we refer to \cite{Ba,Ca3,KS}.

  Recall that, for $p\in\mathbb N_0$, the generalized Dedekind sum is defined by \begin{equation}\label{def-GD}
S_p(a,b)=\sum_{j=0}^{b-1}\B_p\left(\frac{aj}{b}\right)\B_1\left(\frac{{j}}{b}\right),
\end{equation}where $\B_p\left(x\right)$ is the $p$-th Bernoulli function defined by \begin{equation}\label{Bernoulli}\B_k(x)=B_{k}(\{x\})~\textrm{for}~k > 1~\textrm{and}~ \B_1(x)=((x)).\end{equation}
Here $B_{n}(x)$ denotes the $n$-th Bernoulli polynomial and $((x))$ denotes
$$((x))=\begin{cases}
x-[x]-\frac12 &\text{if } x\not\in \mathbb Z, \\
0 &\text{otherwise} \end{cases}$$
(see \cite{Ber2,CK,Har,Si,Sit,Ta}).
More than 60 years ago, Apostol~\cite{Ap} proved a reciprocity formula for this sum.

The Apostol-Dedekind sum $T_p(a,b)$
with quasi-periodic Euler functions is defined by
\begin{equation}\label{def-DS}
T_p(a,b)=2\sum_{j=0}^{b-1}(-1)^j\E_p\left(\frac{aj}{b}\right)\E_1\left(\frac{{j}}{b}\right),
\end{equation}
which is an analogue of the generalized Dedekind sums (\ref{def-GD}) for quasi-periodic Euler functions.

Notice that, as indicated by Carlitz in~\cite[p.661, 2nd paragraph]{Ca3}, the Bernoulli function are periodic, but the Euler functions are just quasi-periodic (also comparing with Eqs. (\ref{ae-ft-p}) and (\ref{Bernoulli}) above), thus the signs $(-1)^{j}$ in the definition of the Apostol-Dedekind sum $T_p(a,b)$
with quasi-periodic Euler functions (\ref{def-DS}) are necessary.

In this paper, we shall  prove the following reciprocity formula for the sum $T_p(a,b).$

\begin{theorem}[Reciprocity formula]\label{re-thm}
Let $a$ and $b$ be positive odd integers with $(a,b)=1.$
\begin{enumerate}
\item[(1)] For a even positive integer $p,$ we have
$$\begin{aligned}
ab^{p+1}T_p(a,b)+a^{p+1}bT_p(b,a)&=2E_{p+1}(0)-ab\sum_{k=0}^{p-1}\binom{p}{k}b^{k}E_{k}(0)a^{p-k}E_{p-k}(0) \\
&=2E_{p+1}(0)-ab(bE(0)+aE(0))^p~(\text{symbolically}).
\end{aligned}$$
\item[(2)] For an odd positive integer $p,$ we have
$$b^{p}T_p(a,b)-a^{p}T_p(b,a)=(a^p-b^p)E_p(0).$$
\end{enumerate}
\end{theorem}
\begin{remark}By applying the Euler-MacLaurin summation formula and its generalization for Bernoulli functions~\cite{Ben}, Da\u{g}l{\i} and Can \cite{DM2,DM} gave new proofs for the reciprocity formulas of the generalized Dedekind sums and the generalized Hardy-Berndt sums. In this note, we shall modify their methods to prove our results by using  the Boole summation formula.

The Boole summation formula is the Euler polynomial version of  the well-known Euler-MacLaurin summation formula (see Lemma \ref{BSF} below), as pointed out by N.E. N\"orlund (\cite{No}), this formula is also due to Euler.\end{remark}

\begin{remark}The proof of reciprocity formula for   another type of Dedekind sum was given in \cite[Theorem 9]{KT}
and  \cite[Theorem 13]{Si2}.
For a general treatment of reciprocity formula for the sum $T_p(a,b)$,
we refer to \cite[Theorem 1.1]{KS}.\end{remark}
In the last section, server examples will be shown.

\section{Proof of Theorem \ref{re-thm}}
We need the following lemmas.

\begin{lemma}[Fourier expansion, {\cite[p.~805, 23.1.16]{AS} and \cite[Lamma 5]{Si2}}]\label{ep-pro} We have
$$\E_p(x)=\frac{4p!}{\pi^{p+1}}\sum_{k=0}^\infty\frac{\sin((2k+1)\pi x-\frac12\pi p)}{(2k+1)^{p+1}},$$
where $0\leq x<1$ if $p\in\mathbb N$ and $0<x<1$ if $p=0.$
\end{lemma}

\begin{lemma}[{\cite[p.~355, (1.2.13)]{Ba}}]\label{lem-1}
If $b$ is an odd positive integer, we have
$$
\sum_{j=0}^{b-1}(-1)^j\E_p\left(\frac{x+j}{b}\right)=b^{-p}\E_p(x)
$$
for $p\in\mathbb N_0$ and arbitrary real numbers $x.$
\end{lemma}

\begin{proposition}
For even $a$  and odd $b,q$, we have
$$T_p(qa,qb)=q^{-1} T_p(a,b).$$
\end{proposition}
\begin{proof}
For even $a$  and odd $b,q$, by (\ref{ae-ft-p}) and Lemma \ref{lem-1}, we have
$$\begin{aligned}
T_p(qa,qb)&=2\sum_{j=0}^{qb-1}(-1)^j\E_p\left(\frac{qaj}{qb}\right)\E_1\left(\frac{j}{qb}\right) \\
&=2\sum_{r=0}^{b-1}\sum_{l=0}^{q-1}(-1)^l(-1)^r\E_p\left(\frac{ar}{b}+al\right)\E_1\left(\frac{r/b+l}{q}\right) \\
&=2\sum_{r=0}^{b-1}(-1)^r\E_p\left(\frac{ar}{b}\right)\sum_{l=0}^{q-1}(-1)^l\E_1\left(\frac{r/b+l}{q}\right) \\
&=q^{-1}T_p(a,b).
\end{aligned}$$
 This completes our proof.
\end{proof}

\begin{lemma}\label{lem-s3}
For odd positive  integers $a$ and $b$ with $(a,b)=1,$ we have
$$
\sum_{j=0}^{b-1}(-1)^j\E_p\left(\frac{x+aj}{b}\right)=b^{-p}\E_p(x)
$$
for $p\in\mathbb N_0$ and arbitrary real numbers $x.$
\end{lemma}
\begin{proof}
For $n\in\mathbb N_0,$  $\{n\}_N$ denotes by be the integer between 0 and $N-1$
which is congruent to $n$ modulo ${N}.$
Let $a$ and $b$ be odd positive integers with $(a,b)=1.$
Note that
$$aj\equiv\{aj\}_b\pmod{b}$$
and
$$\frac{\{aj\}_b}{b}=\frac{aj}{b}-\left[\frac{aj}{b}\right]\quad\text{($[\;]=$ greatest integer function)}.$$
Hence, by (\ref{ae-ft-p}) and Lemma \ref{lem-1}, we have
\begin{equation}\label{rab-form}
\begin{aligned}
\sum_{j=0}^{b-1}(-1)^j\E_p\left(\frac xb+\frac{aj}{b}\right)
&=\sum_{j=0}^{b-1}(-1)^j\E_p\left(\frac xb+\frac{\{aj\}_b}{b}+\left[\frac{aj}{b}\right]\right) \\
&=\sum_{j=0}^{b-1}(-1)^{j+\left[\frac{aj}{b}\right]}\E_p\left(\frac xb+\frac{\{aj\}_b}{b}\right) \\
&=\sum_{j'=0}^{b-1}(-1)^{j'}\E_p\left(\frac xb+\frac{j'}{b}\right) \\
&=b^{-p}\E_p(x),
\end{aligned}
\end{equation}
where the following fact has been used ``there exist unique $j'\in\{0,1,\ldots,b-1\}$ such that $j'=\{aj\}_b$ for each $j=0,1,\ldots,b-1,$
and $j+\left[\frac{aj}{b}\right],j'$ have the same parity."
\end{proof}

The following is the Boole summation formula (see, for example, \cite[24.17.1--2]{NIST}).

\begin{lemma}[Boole summation formula]\label{BSF}
Let $\alpha,\beta$ and $m$ be integers such that $\alpha<\beta$ and $0<m.$ If $f^{(m)}(x)$ is absolutely integrable over $[\alpha,\beta].$
Then
$$\begin{aligned}
2\sum_{j=\alpha}^{\beta-1}(-1)^jf(j)& = \sum_{k=0}^{m-1}\frac{E_k(0)}{k!}\left((-1)^{\beta-1}f^{(k)}(\beta)+(-1)^{\alpha}f^{(k)}(\alpha) \right) \\
&\quad+
\frac1{(m-1)!}\int_{\alpha}^{\beta}f^{(m)}(x)\E_{m-1}(-x)dx.
\end{aligned}$$
\end{lemma}

The above formula is proved by Boole \cite{Boole}, but a similar one may be known  by Euler as well (see \cite{No}).

\begin{lemma}\label{lem-2}
For odd positive  integers $a$ and $b$ with $(a,b)=1$ and $p,n\in\mathbb N_0$, we have
$$
\int_0^1\E_p(ax)\E_{n}(bx)dx=\frac{2(-1)^{n+1}}{(n+1)\binom{p+n+1}{n+1}}\frac1{a^{n+1}b^{p+1}}E_{p+n+1}(0).
$$
\end{lemma}
\begin{proof}
First, we consider the function $f(x)=\E_p(xy),$ where $y\in\mathbb R.$
The property
\begin{equation}\label{part-eft}
\frac{\text{d}}{\text{d}x}(E_{p}(x))=pE_{p-1}(x), \quad p\in\mathbb N
\end{equation}
implies
\begin{equation}\label{part-eft-1}
\frac{\text{d}^j}{\text{d}x^j}(f(x))=\frac{\text{d}^j}{\text{d}x^j}(\E_p(xy))=y^j\frac{p!}{(p-j)!}\E_{p-j}(xy)
\end{equation}
for $1\leq j\leq p.$ Set $\alpha=0$ and $\beta=b$ in Lemma \ref{BSF}. From (\ref{part-eft-1}), Lemma \ref{BSF} can be written as
\begin{equation}\label{part-eft-2}
\begin{aligned}
2\sum_{j=0}^{b-1}(-1)^j\E_p(jy)& = \sum_{k=0}^{m-1}E_k(0)y^k\binom pk\left((-1)^{b-1}\E_{p-k}(by)
+\E_{p-k}(0) \right) \\
&\quad+
m\binom pm y^m\int_{0}^{b}\E_{p-m}(xy)\E_{m-1}(-x)dx.
\end{aligned}
\end{equation}
Letting $a,b$ be odd positive  integers with $(a,b)=1,$ and $y=a/b$ in (\ref{part-eft-2}).
Since $\E_{p-k}(a)=(-1)^a\E_{p-k}(0)=-E_{p-k}(0),$ we obtain
\begin{equation}\label{part-eft-3}
\begin{aligned}
2\sum_{j=0}^{b-1}(-1)^j\E_p\left(\frac{aj}{b}\right)=
m\left(\frac ab\right)^m \binom pm\int_{0}^{b}\E_{p-m}\left(\frac{ax}{b}\right)\E_{m-1}(-x)dx.
\end{aligned}
\end{equation}
Let $k=0$ in (\ref{ae-ft-1}), we obtain $$\lim_{x\rightarrow0^+}\E_0(1-x)=1=\lim_{x\rightarrow0^+}(-1)^0\E_0(x).$$
By Lemma \ref{ep-pro}, we have $\E_k(1-x)=(-1)^k\E_k(x)$, for $k\in\mathbb N.$
Thus, $\E_k(1-x)=(-1)^k\E_k(x)$, for $k\in\mathbb N_{0}.$
So, for $m\in\mathbb N,$ we have
\begin{equation}\label{part-eft-4}
\begin{aligned}
\int_{0}^{b}\E_{p-m}\left(\frac{ax}{b}\right)\E_{m-1}(-x)dx
&=-\int_{0}^{b}\E_{p-m}\left(\frac{ax}{b}\right)\E_{m-1}(1-x)dx \\
&=(-1)^m\int_{0}^{b}\E_{p-m}\left(\frac{ax}{b}\right)\E_{m-1}(x)dx
\end{aligned}
\end{equation}
by (\ref{ae-ft}).
Letting $x=bt$ in (\ref{part-eft-4}), we obatin
\begin{equation}\label{part-eft-5}
\begin{aligned}
\int_{0}^{b}\E_{p-m}\left(\frac{ax}{b}\right)\E_{m-1}(-x)dx
=(-1)^mb\int_0^1\E_{p-m}(at)\E_{m-1}(bt)dt.
\end{aligned}
\end{equation}
Substituting the above  into (\ref{part-eft-3}), we have
\begin{equation}\label{part-eft-6}
\begin{aligned}
\int_0^1\E_{p-m}(ax)\E_{m-1}(bx)dx=\frac{2(-1)^m}{m\binom pm}\frac{b^{m-1}}{a^m } \sum_{j=0}^{b-1}(-1)^j\E_p\left(\frac{aj}{b}\right).
\end{aligned}
\end{equation}
Therefore, using Lemma \ref{lem-s3} with $x=0,$ (\ref{part-eft-6}) can be written as
\begin{equation}\label{part-eft-7}
\int_0^1\E_{p-m}(ax)\E_{m-1}(bx)dx=\frac{2(-1)^m}{m\binom pm}\frac{b^{m-p-1}}{a^m }E_p(0).
\end{equation}
Setting $p=p+m$ in (\ref{part-eft-7}), we obtain the desired result.
\end{proof}

\begin{proof}[Proof of Theorem \ref{re-thm}]
Set $f(x)=x\E_p(xy),$ where $y\in\mathbb R.$

From (\ref{part-eft}) and Leibniz's rule for
the derivative, we have
\begin{equation}\label{re-pf-1}
\frac{\text{d}^j}{\text{d}x^j}(f(x))=xy^j\frac{p!}{(p-j)!}\E_{p-j}(xy)+y^{j-1}j\frac{p!}{(p+1-j)!}\E_{p+1-j}(xy),
\end{equation}
for $1\leq j\leq p,$

Set $\alpha=0$ and $\beta=b$ in Lemma \ref{BSF}.
Let $a,b$ be positive odd integers with $(a,b)=1.$
From (\ref{re-pf-1}) with $y=a/b,$ Lemma \ref{BSF} can be written as
\begin{equation}\label{re-pf-2}
\begin{aligned}
2\sum_{j=0}^{b-1}(-1)^jj\E_p\left(\frac{aj}{b}\right)& = \sum_{k=0}^{m-1}E_k(0)
\biggl[\left(\frac ab\right)^{k}\binom pk b\E_{p-k}(a) \\
&\quad
+\left(\frac ab\right)^{k-1}\binom{p+1}{k}\frac{k}{p+1}\E_{p+1-k}(a) \\
&\quad+\left(\frac ab\right)^{k-1}\binom{p+1}{k}\frac{k}{p+1}\E_{p+1-k}(0)
\biggl] \\
&\quad
+\int_{0}^{b}
\biggl[\left(\frac ab\right)^{m}\binom pm mx\E_{p-m}\left(\frac{ax}b\right) \\
&\quad
+\left(\frac ab\right)^{m-1}\binom{p}{m-1}m\E_{p+1-m}\left(\frac{ax}b\right)\biggl]
\E_{m-1}(-x)dx.
\end{aligned}
\end{equation}
Notice that $\E_{p+1-k}(a)=(-1)^a\E_{p+1-k}(0)=-E_{p+1-k}(0),$
(\ref{re-pf-2}) becomes
\begin{equation}\label{re-pf-3}
\begin{aligned}
2\sum_{j=0}^{b-1}(-1)^jj\E_p\left(\frac{aj}{b}\right)& = -b\sum_{k=0}^{m-1}\binom pk
\left(\frac ab\right)^{k}E_k(0)E_{p-k}(0) \\
&\quad+\left(\frac ab\right)^{m}\binom pm m\int_{0}^{b}
x\E_{p-m}\left(\frac{ax}b\right)\E_{m-1}(-x)dx\\
&\quad+\left(\frac ab\right)^{m-1}\binom{p}{m-1}m\int_{0}^{b}
\E_{p+1-m}\left(\frac{ax}b\right)\E_{m-1}(-x)dx.
\end{aligned}
\end{equation}
From Lemma \ref{lem-2}, the second integral term in (\ref{re-pf-3}) can be written as
\begin{equation}\label{re-pf-4}
\begin{aligned}
\int_{0}^{b}\E_{p+1-m}\left(\frac{ax}{b}\right)\E_{m-1}(-x)dx
&=(-1)^mb\int_0^1\E_{p+1-m}(ax)\E_{m-1}(bx)dx \\
&=\frac{2}{m \binom{p+1}m}\frac{b^{m-p-1}}{a^m}E_{p+1}(0).
\end{aligned}
\end{equation}
We also note that
\begin{equation}\label{re-pf-4-1}
\begin{aligned}
\int_{0}^{b} x\E_{p-m}\left(\frac{ax}b\right)\E_{m-1}(-x)dx
&=(-1)^mb^2\int_{0}^1 x\E_{p-m}(ax)\E_{m-1}(bx)dx.
\end{aligned}
\end{equation}
Therefore, by (\ref{re-pf-4}) and (\ref{re-pf-4-1}), (\ref{re-pf-3}) can be written as
\begin{equation}\label{re-pf-5}
\begin{aligned}
2\sum_{j=0}^{b-1}(-1)^jj\E_p\left(\frac{aj}{b}\right)& = -b\sum_{k=0}^{m-1}\binom pk
\left(\frac ab\right)^{k}E_k(0)E_{p-k}(0) \\
&\quad+\frac{2m}{(p+1)ab^p}E_{p+1}(0) \\
&\quad+\frac{(-1)^m m a^m}{b^{m-2}}\binom pm
\int_{0}^1 x\E_{p-m}(ax)\E_{m-1}(bx)dx.
\end{aligned}
\end{equation}
Let $m=2$ in (\ref{re-pf-5}), since $E_1(0)=-1/2,$ we obtain
\begin{equation}\label{re-pf-6}
\begin{aligned}
2\sum_{j=0}^{b-1}(-1)^jj\E_p\left(\frac{aj}{b}\right)& = -b\left(E_p(0)-\frac{ap}{2b}E_{p-1}(0)\right)+\frac{4}{(p+1)ab^p}E_{p+1}(0) \\
&\quad+{2 a^2}\binom p2
\int_{0}^1 x\E_{p-2}(ax)\E_{1}(bx)dx.
\end{aligned}
\end{equation}
On the other hand, setting $m=p-1,$ and interchanging $a$ and $b$ in (\ref{re-pf-5}), we obtain
\begin{equation}\label{re-pf-7}
\begin{aligned}
2\sum_{j=0}^{a-1}(-1)^jj\E_p\left(\frac{bj}{a}\right)& = -a\sum_{k=0}^{p-2}\binom pk
\left(\frac ba\right)^{k}E_k(0)E_{p-k}(0) \\
&\quad+\frac{2(p-1)}{(p+1)a^pb}E_{p+1}(0) \\
&\quad+\frac{(-1)^{p-1} (p-1) b^{p-1}}{a^{p-3}}\binom p{p-1}
\int_{0}^1 x\E_{1}(bx)\E_{p-2}(ax)dx.
\end{aligned}
\end{equation}
Particularly, since $\E_1(x)=x-1/2$ for $0<x<1,$ we have
\begin{equation}\label{pf-def}
T_p(a,b)=\frac2b\sum_{j=0}^{b-1}(-1)^jj\E_p\left(\frac{aj}{b}\right)-\sum_{j=0}^{b-1}(-1)^j\E_p\left(\frac{aj}{b}\right).
\end{equation}
Thus, with the help of Lemma \ref{lem-s3}, we have
\begin{equation}\label{pf-def-1}
T_p(a,b)=\frac2b\sum_{j=0}^{b-1}(-1)^jj\E_p\left(\frac{aj}{b}\right)-b^{-p}E_p(0).
\end{equation}

For Part (1),
if $p$ is a positive even integer, then form (\ref{def-DS}), (\ref{re-pf-6}), (\ref{re-pf-7}) and (\ref{pf-def-1}),
we obtain the reciprocity formula
$$\begin{aligned}
ab^{p+1}T_p(a,b)+a^{p+1}bT_p(b,a)&=\frac{a^2b^pp}{2}E_{p-1}(0)+2E_{p+1}(0) \\
&\quad-a^{p+1}b\sum_{k=0}^{p-2}\binom{p}{k}\left(\frac ba\right)^{k}E_{k}(0)E_{p-k}(0) \\
&=2E_{p+1}(0)-ab\sum_{k=0}^{p-1}\binom{p}{k}b^{k}E_{k}(0)a^{p-k}E_{p-k}(0),
\end{aligned}$$
which may be written symbolically in the form
$$ab^{p+1}T_p(a,b)+a^{p+1}bT_p(b,a)=2E_{p+1}(0)-ab(bE(0)+aE(0))^p,$$
since $E_1(0)=-1/2$, and $E_p(0)=0$ if $p$ is even. Thus, we have Part (1).

When $p=1,$  consider the sum
\begin{equation}\label{def-DS-p=1}
T_1(a,b)=2\sum_{j=0}^{b-1}(-1)^j\E_1\left(\frac{aj}{b}\right)\E_1\left(\frac{j}{b}\right).
\end{equation}
Set $m=p=1$ in (\ref{re-pf-5}). Since $E_2(0)=0,$ we have
\begin{equation}\label{pf-p=1}
2\sum_{j=0}^{b-1}(-1)^j j\E_1\left(\frac{aj}{b}\right)=-bE_1(0)-ab\int_0^1 x\E_0(ax)\E_0(bx)dx,
\end{equation}
thus from (\ref{pf-def-1}) with $p=1$, we obtain
\begin{equation}\label{pf-p=1-2}
T_1(a,b)=-E_1(0)-\frac1b E_1(0)-a\int_0^1 x\E_0(ax)\E_0(bx)dx.
\end{equation}
Also setting $m=p=1,$ and interchanging $a$ and $b$ in (\ref{re-pf-5}), we have:
\begin{equation}\label{pf-p=1-3}
2\sum_{j=0}^{a-1}(-1)^j j\E_1\left(\frac{bj}{a}\right)=-aE_1(0)-ba\int_0^1 x\E_0(bx)\E_0(ax)dx,
\end{equation}
thus from (\ref{pf-def-1}) with $p=1$, we obtain
\begin{equation}\label{pf-p=1-4}
T_1(b,a)=-E_1(0)-\frac1a E_1(0)-b\int_0^1 x\E_0(ax)\E_0(bx)dx.
\end{equation}
Therefore, by (\ref{pf-p=1-2}) and (\ref{pf-p=1-4}), we have
\begin{equation}\label{pf-p=1-5}
bT_1(a,b)-aT_1(b,a)=(a-b)E_1(0).
\end{equation}

Note that, since $E_k(0)=0$ for even $k$, we obtain
\begin{equation}\label{re-pf-8}
\begin{aligned}
\sum_{k=0}^{p-2}\binom{p}{k}\left(\frac ba\right)^{k}E_{k}(0)E_{p-k}(0)=E_p(0) \quad\text{if odd }p>2.
\end{aligned}
\end{equation}
If $p$ is a positive odd integer with $p>2,$ then form (\ref{def-DS}), (\ref{re-pf-6}), (\ref{re-pf-7}), (\ref{pf-def-1}) and (\ref{re-pf-8}),
we obtain the reciprocity formula
$$b^{p}T_p(a,b)-a^{p}T_p(b,a)=(a^p-b^p)E_p(0).$$
This completes the proof of Part (2) for odd integer $p\geq1.$
\end{proof}

\begin{remark}
Let $a$ and $b$ be positive odd integers with $(a,b)=1.$
From (\ref{ae-ft-p}), (\ref{fr-in}) and Lemma \ref{lem-s3}, we have
$$
\begin{aligned}
\sum_{j=0}^{b-1}(-1)^{j+\left[\frac{aj}{b}\right]}
&=\sum_{j=0}^{b-1}(-1)^j(-1)^{\left[\frac{aj}{b}\right]}\E_0\left(\left\{ \frac{aj}{b} \right\}\right) \\
&=\sum_{j=0}^{b-1}(-1)^j\E_0\left(\left\{ \frac{aj}{b} \right\}+\left[\frac{aj}{b}\right]\right) \\
&=\sum_{j=0}^{b-1}(-1)^j\E_0\left(\frac{aj}{b}\right) \\
&=E_0(0)=1.
\end{aligned}
$$
Thus, let $\varrho(a,b)=\sum_{j=0}^{b-1}(-1)^{j+\left[\frac{aj}{b}\right]},$ we have
$$\varrho(a,b)+\varrho(b,a)=2$$
(also see  \cite[Theorem 4.2]{Ber3}).
\end{remark}

\begin{remark}
Let $I=\int_0^1 (-1)^{[ax]+[bx]}dx,$ where $a$ and $b$ are positive odd integers with $(a,b)=1.$
From Lemma \ref{lem-2}, we have
$$I=\int_0^1 (-1)^{[ax]+[bx]}dx=\int_0^1 \E_0(ax)\E_0(bx)dx=\frac1{ab},$$
for $(a,b)=1$.
This integral can be calculated by Stieltjes integrals.
For $(a,b)=1,$ we see that
$$I=\sum_{j=0}^{a-1}\int_{j/a}^{(j+1)/a} \E_0(ax)\E_0(bx)dx.$$
Thus by substituting  $x=y/a+j/a,$ and $dx=dy/a$, we get
$$
\begin{aligned}
I&=\frac1a\sum_{j=0}^{a-1}\int_{0}^{1} \E_0(y+j)\E_0\left(\frac{by}{a}+\frac{bj}{a}\right)dy  \\
&=\frac1a\sum_{j=0}^{a-1}(-1)^j\int_{0}^{1} \E_0(y)\E_0\left(\frac{by}{a}+\frac{bj}{a}\right)dy \\
&=\frac1a\int_{0}^{1}\sum_{j=0}^{a-1}(-1)^j\E_0\left(\frac{by}{a}+\frac{bj}{a}\right)dy \\
&=\frac1a\int_{0}^{1}\E_0(by)dy.
\end{aligned}
$$
Here we use Lemma \ref{lem-s3}. Repeating the procedure on $b,$ we obtain
$$\int_{0}^{1}\E_0(by)dy=\frac1b$$
and $I=1/ab$ as claimed.
\end{remark}

\begin{remark}
From (\ref{def-DS-p=1}), we have $T_1(a,1)=1/2,$ so when $b=1,$ (\ref{pf-p=1-2}) may be written in the form
$$\frac12=1-a\int_0^1 x\E_0(ax)dx,$$
since $E_1(0)=-1/2.$ An immediate consequence of this formula is
\begin{equation}\label{x-int}
\int_0^1 (-1)^{[ax]}xdx=\int_0^1 x\E_0(ax)dx=\frac1{2a}.
\end{equation}
On the other hand, from (\ref{pf-p=1-2}), we see that
$$
\begin{aligned}
T_1(1,a)&=\frac12+\frac1{2a}-\int_0^1 x\E_0(x)\E_0(ax)dx \\
&=\frac12+\frac1{2a}-\int_0^1 x\E_0(ax)dx \\
&=\frac12,
\end{aligned}
$$
where (\ref{x-int}) has been used. Therefore, we obtain
$T_1(a,1)-aT_1(1,a)=(a-1)E_1(0)$
(also see Theorem \ref{re-thm}(2) above).
\end{remark}

\section{Some examples}

We have the following concrete examples of the reciprocity formula (Theorem \ref{re-thm}).

\begin{example}
Set $a=5$ and $b=3$ in (\ref{def-DS}), by (\ref{ae-ft-p}), we have
\begin{equation}\label{ex2}
\begin{aligned}
T_p(5,3)&=2\left(\E_p\left(0\right)\E_1\left(0\right) -\E_p\left(\frac{5}{3}\right)\E_1\left(\frac{1}{3}\right)+\E_p\left(\frac{10}{3}\right)\E_1\left(\frac{2}{3}\right)\right) \\
&=2\left(E_p\left(0\right)E_1\left(0\right) + E_p\left(\frac{2}{3}\right)E_1\left(\frac{1}{3}\right)-E_p\left(\frac{1}{3}\right)E_1\left(\frac{2}{3}\right)\right)
\end{aligned}
\end{equation}
and
\begin{equation}\label{ex1}
\begin{aligned}
T_p(3,5)&=2\biggl(\E_p\left(0\right)\E_1\left(0\right) -\E_p\left(\frac{3}{5}\right)\E_1\left(\frac{1}{5}\right)+\E_p\left(\frac{6}{5}\right)\E_1\left(\frac{2}{5}\right)\\
&\quad-\E_p\left(\frac{9}{5}\right)\E_1\left(\frac{3}{5}\right)
+\E_p\left(\frac{12}{5}\right)\E_1\left(\frac{4}{5}\right)\biggl) \\
&=2\biggl(E_p\left(0\right)E_1\left(0\right)  -E_p\left(\frac{3}{5}\right)E_1\left(\frac{1}{5}\right)-E_p\left(\frac{1}{5}\right)\E_1\left(\frac{2}{5}\right)
 \\
&\quad+E_p\left(\frac{4}{5}\right)\E_1\left(\frac{3}{5}\right)+E_p\left(\frac{2}{5}\right)\E_1\left(\frac{4}{5}\right)\biggl).
\end{aligned}
\end{equation}

Case (I): $p=1.$ Letting $p=1$ in (\ref{ex2}) and (\ref{ex1}), we have
$$T_1(5,3)=\frac{1}{2},\quad T_1(3,5)=\frac{1}{2},$$
so the left-hand side of Theorem \ref{re-thm}(2) reduces to
\begin{equation}\label{ex5-I}
\begin{aligned}
3T_1(5,3)-5 T_1(3,5)=-1
\end{aligned}
\end{equation}
and if $p=1$, then the right-hand side of Theorem \ref{re-thm}(2) equals to
\begin{equation}\label{ex6-I}
\begin{aligned}
(5-3)E_1(0)=-1.
\end{aligned}
\end{equation}
Therefore, (\ref{ex5-I}) and (\ref{ex6-I}) yield the result of Theorem \ref{re-thm}(2) when $p=1,$ $a=5$ and $b=3.$

Case (II): $p=2.$ Letting $p=2$ in (\ref{ex2}) and (\ref{ex1}), we have
$$T_2(5,3)=\frac{4}{27},\quad T_2(3,5)=-\frac{44}{125},$$
so the left-hand side of Theorem \ref{re-thm}(1) reduces to
\begin{equation}\label{ex3-II}
\begin{aligned}
5\cdot3^3T_2(5,3)+5^3\cdot3 T_2(3,5)=-112
\end{aligned}
\end{equation}
and if $p=2$, then the right-hand side of Theorem \ref{re-thm}(1) equals to
\begin{equation}\label{ex4-II}
\begin{aligned}
2E_3(0)-&5\cdot3\sum_{k=0}^1\binom 2k5^kE_k(0)3^{2-k}E_{2-k}(0) \\
&=2E_3(0)-3^2\cdot5^2\binom21 E_1(0)E_1(0) \\
&=-112,
\end{aligned}
\end{equation}
since $E_2(0)=0.$ Therefore, (\ref{ex3-II}) and (\ref{ex4-II}) yields the result of Theorem \ref{re-thm}(1) when $p=2,$ $a=5$ and $b=3.$

Case (III): $p=3.$ Letting $p=3$ in (\ref{ex2}) and (\ref{ex1}), we have
$$T_3(5,3)=-\frac{1}{4},\quad T_3(3,5)=-\frac{1}{4},$$
so the left-hand side of Theorem \ref{re-thm}(2) reduces to
\begin{equation}\label{ex5}
\begin{aligned}
3^3T_3(5,3)-5^3 T_3(3,5)=\frac{49}{2}
\end{aligned}
\end{equation}
and if $p=3$, the right-hand side of Theorem \ref{re-thm}(2) equals to
\begin{equation}\label{ex6}
\begin{aligned}
(5^3-3^3)E_3(0)=\frac{49}{2}.
\end{aligned}
\end{equation}
Therefore, (\ref{ex5}) and (\ref{ex6}) yield the result of Theorem \ref{re-thm}(2) when $p=3,$ $a=5$ and $b=3.$

Case (IV): $p=4.$ Letting $p=4$ in (\ref{ex2}) and (\ref{ex1}), we have
$$T_4(5,3)=-\frac{44}{243},\quad T_4(3,5)=\frac{1348}{3125},$$
so the left-hand side of Theorem \ref{re-thm}(1) reduces to
\begin{equation}\label{ex3}
\begin{aligned}
5\cdot3^5T_4(5,3)+5^5\cdot3 T_4(3,5)=3824
\end{aligned}
\end{equation}
and if $p=4$, then the right-hand side of Theorem \ref{re-thm}(1) equals to
\begin{equation}\label{ex4}
\begin{aligned}
2E_5(0)-&5\cdot3\sum_{k=0}^3\binom 4k5^kE_k(0)3^{4-k}E_{4-k}(0) \\
&=2E_5(0)-5\cdot3(4\cdot5\cdot3^3+4\cdot5^3\cdot3)E_1(0)E_3(0) \\
&=3824,
\end{aligned}
\end{equation}
since $E_2(0)=E_4(0)=0.$ Therefore, (\ref{ex3}) and (\ref{ex4}) yield the result of Theorem \ref{re-thm}(1) when $p=4,$ $a=5$ and $b=3.$

Case (V): $p=5.$ By letting $p=5$ in (\ref{ex2}) and (\ref{ex1}), we have
$$T_5(5,3)=\frac{1}{2},\quad T_5(3,5)=\frac{1}{2},$$
so the left-hand side of Theorem \ref{re-thm}(2) reduces to
\begin{equation}\label{ex5-V}
\begin{aligned}
3^5T_5(5,3)-5^5 T_5(3,5)=-1441
\end{aligned}
\end{equation}
and if $p=5,$ then the right-hand side of Theorem \ref{re-thm}(2) for $p=5$ equals to
\begin{equation}\label{ex6-V}
\begin{aligned}
(5^5-3^5)E_5(0)=-1441.
\end{aligned}
\end{equation}
Therefore, (\ref{ex5-V}) and (\ref{ex6-V}) yield the result of Theorem \ref{re-thm}(2) when $p=5,$ $a=5$ and $b=3.$
\end{example}

\begin{example}
Set $a=7$ and $b=11$ in (\ref{def-DS}). From (\ref{ae-ft-p}), we have
\begin{equation}\label{ex7}
\begin{aligned}
T_p(7,11)&=2\biggl( \E_p\left(0\right)\E_1\left(0\right) -\E_p\left(\frac{7}{11}\right)\E_1\left(\frac{1}{11}\right)
+\E_p\left(\frac{14}{11}\right)\E_1\left(\frac{2}{11}\right) \\
&\quad
-\E_p\left(\frac{21}{11}\right)\E_1\left(\frac{3}{11}\right)+\E_p\left(\frac{28}{11}\right)\E_1\left(\frac{4}{11}\right) \\
&\quad-\E_p\left(\frac{35}{11}\right)\E_1\left(\frac{5}{11}\right)+\E_p\left(\frac{42}{11}\right)\E_1\left(\frac{6}{11}\right) \\
&\quad-\E_p\left(\frac{49}{11}\right)\E_1\left(\frac{7}{11}\right)+\E_p\left(\frac{56}{11}\right)\E_1\left(\frac{8}{11}\right) \\
&\quad-\E_p\left(\frac{63}{11}\right)\E_1\left(\frac{9}{11}\right)+\E_p\left(\frac{70}{11}\right)\E_1\left(\frac{10}{11}\right)
\biggl) \\
&=2\biggl( E_p\left(0\right)E_1\left(0\right) -E_p\left(\frac{7}{11}\right)E_1\left(\frac{1}{11}\right)-E_p\left(\frac{3}{11}\right)E_1\left(\frac{2}{11}\right) \\
&\quad+E_p\left(\frac{10}{11}\right)E_1\left(\frac{3}{11}\right)+E_p\left(\frac{6}{11}\right)E_1\left(\frac{4}{11}\right) \\
&\quad+E_p\left(\frac{2}{11}\right)E_1\left(\frac{5}{11}\right)-E_p\left(\frac{9}{11}\right)E_1\left(\frac{6}{11}\right) \\
&\quad-E_p\left(\frac{5}{11}\right)E_1\left(\frac{7}{11}\right)-E_p\left(\frac{1}{11}\right)E_1\left(\frac{8}{11}\right) \\
&\quad+E_p\left(\frac{8}{11}\right)E_1\left(\frac{9}{11}\right)+E_p\left(\frac{4}{11}\right)E_1\left(\frac{10}{11}\right)
\biggl)
\end{aligned}
\end{equation}
and
\begin{equation}\label{ex8}
\begin{aligned}
T_p(11,7)&=2\biggl( \E_p\left(0\right)\E_1\left(0\right)-\E_p\left(\frac{11}{7}\right)\E_1\left(\frac{1}{7}\right)
+\E_p\left(\frac{22}{7}\right)\E_1\left(\frac{2}{7}\right) \\
&\quad-\E_p\left(\frac{33}{7}\right)\E_1\left(\frac{3}{7}\right)+\E_p\left(\frac{44}{7}\right)\E_1\left(\frac{4}{7}\right) \\
&\quad-\E_p\left(\frac{55}{7}\right)\E_1\left(\frac{5}{7}\right)+\E_p\left(\frac{66}{7}\right)\E_1\left(\frac{6}{7}\right)
 \biggl)\\
&=2\biggl( E_p\left(0\right)E_1\left(0\right)
+E_p\left(\frac{4}{7}\right)E_1\left(\frac{1}{7}\right)-E_p\left(\frac{1}{7}\right)E_1\left(\frac{2}{7}\right) \\
&\quad-E_p\left(\frac{5}{7}\right)E_1\left(\frac{3}{7}\right)+E_p\left(\frac{2}{7}\right)E_1\left(\frac{4}{7}\right) \\
&\quad+E_p\left(\frac{6}{7}\right)E_1\left(\frac{5}{7}\right)-E_p\left(\frac{3}{7}\right)E_1\left(\frac{6}{7}\right)
\biggl).
\end{aligned}
\end{equation}

Case (I): $p=1.$ Letting $p=1$ in (\ref{ex7}) and (\ref{ex8}), we have
$$T_1(7,11)=\frac{1}{2},\quad T_1(11,7)=\frac{1}{2},$$
so the left-hand side of Theorem \ref{re-thm}(2) reduces to
\begin{equation}\label{ex11-I}
\begin{aligned}
11T_1(7,11)-7T_1(11,7)=2
\end{aligned}
\end{equation}
and if $p=1$, then the right-hand side of Theorem \ref{re-thm}(2)   equals to
\begin{equation}\label{ex12-I}
\begin{aligned}
(7-11)E_1(0)=2.
\end{aligned}
\end{equation}
Therefore, (\ref{ex11-I}) and (\ref{ex12-I})  yield the result of Theorem \ref{re-thm}(2) when $p=1,$ $a=7$ and $b=11.$

Case (II): $p=2.$ Letting $p=2$ in (\ref{ex7}) and (\ref{ex8}), we have
$$T_2(7,11)=-\frac{524}{1331},\quad T_2(11,7)=\frac{64}{343},$$
so the left-hand side of Theorem \ref{re-thm}(1) reduces to
\begin{equation}\label{ex9-II}
\begin{aligned}
7\cdot11^3T_2(7,11)+7^3\cdot11 T_2(7,11)= -2964
\end{aligned}
\end{equation}
and if $p=2$, then the right-hand side of Theorem \ref{re-thm}(1) equals to
\begin{equation}\label{ex10-II}
\begin{aligned}
2E_3(0)-&7\cdot11\sum_{k=0}^1\binom 2k11^kE_k(0)7^{2-k}E_{2-k}(0) \\
&=2E_3(0)-7^2\cdot11^2\binom21 E_1(0)E_1(0) \\
&= -2964 ,
\end{aligned}
\end{equation}
since $E_2(0)=0.$ Therefore, (\ref{ex9-II}) and (\ref{ex10-II})
reduce to yield the result of Theorem \ref{re-thm}(1) when $p=2,$ $a=7$ and $b=11.$

Case (III): $p=3.$ Letting $p=3$ in (\ref{ex7}) and (\ref{ex8}), we have
$$T_3(7,11)=-\frac{1}{4},\quad T_3(11,7)=-\frac{1}{4},$$
so the left-hand side of Theorem \ref{re-thm}(2) reduces to
\begin{equation}\label{ex11-III}
\begin{aligned}
11^3T_3(7,11)-7^3T_3(11,7)=-247
\end{aligned}
\end{equation}
and if $p=3$, then the right-hand side of Theorem \ref{re-thm}(2) equals to
\begin{equation}\label{ex12-III}
\begin{aligned}
(7^3-11^3)E_3(0)=-247.
\end{aligned}
\end{equation}
Therefore, (\ref{ex11-III}) and (\ref{ex12-III})  yield the result of Theorem \ref{re-thm}(2) when $p=3,$ $a=7$ and $b=11.$

Case (IV): $p=4.$ Letting even $p=4$ in (\ref{ex7}) and (\ref{ex8}), we have
$$T_4(7,11)=\frac{78532}{161051},\quad T_4(11,7)=-\frac{4160}{16807},$$
so the left-hand side of Theorem \ref{re-thm}(1) reduces to
\begin{equation}\label{ex9}
\begin{aligned}
7\cdot11^5T_4(7,11)+7^5\cdot11 T_4(7,11)=503964
\end{aligned}
\end{equation}
and if $p=3$, then the right-hand side of Theorem \ref{re-thm}(1) equals to
\begin{equation}\label{ex10}
\begin{aligned}
2E_5(0)-&7\cdot11\sum_{k=0}^3\binom 4k11^kE_k(0)7^{4-k}E_{4-k}(0) \\
&=2E_5(0)-7^2\cdot11^2\left(\binom41 7^2E_1(0)E_3(0)+\binom43 11^2 E_3(0)E_1(0)  \right) \\
&= 503964 ,
\end{aligned}
\end{equation}
since $E_2(0)=E_4(0)=0.$ Therefore, (\ref{ex9}) and (\ref{ex10}) yield the result of Theorem \ref{re-thm}(1) when $p=4,$ $a=7$ and $b=11.$

Case (II): $p=5.$ Letting $p=5$ in (\ref{ex7}) and (\ref{ex8}), we have
$$T_5(7,11)=\frac{1}{2},\quad T_5(11,7)=\frac{1}{2},$$
so the left-hand side of Theorem \ref{re-thm}(2) reduces to
\begin{equation}\label{ex11}
\begin{aligned}
11^5T_5(7,11)-7^5 T_5(11,7)=72122
\end{aligned}
\end{equation}
and if $p=5$, then the right-hand side of Theorem \ref{re-thm}(2) equals to
\begin{equation}\label{ex12}
\begin{aligned}
(7^5-11^5)E_5(0)=72122.
\end{aligned}
\end{equation}
Therefore, (\ref{ex11}) and (\ref{ex12}) yield the result of Theorem \ref{re-thm}(2) when $p=5,$ $a=7$ and $b=11.$
\end{example}

\bibliography{central}

\end{document}